\documentclass[preprint,12pt,3p]{elsarticle}

\usepackage{amssymb}

\usepackage{amsthm}
\usepackage{graphicx}
\usepackage{amsmath}
\usepackage{hyperref}
\usepackage{graphicx}

\usepackage{pgf,tikz}

\usetikzlibrary{snakes}
\tikzstyle{printersafe}=[snake=snake,segment amplitude=0 pt]

\usetikzlibrary{arrows}
\usetikzlibrary{shapes}
\usepackage{multirow}
\usepackage{latexsym}
\usepackage{setspace}
\PassOptionsToPackage{boxed,section}{algorithm}
\usepackage{algorithm}
\usepackage{algorithmic}
\usepackage{enumerate}
\usepackage{amsmath}			
\usepackage{amssymb}			
\usepackage{url}
\usepackage{setspace}
\usepackage{tikz}
\usetikzlibrary{arrows}
\usetikzlibrary{shapes}
\usetikzlibrary{decorations.markings}
\usepackage{geometry}
\usepackage{bbm}

\newtheorem{problem}{\em Problem}
\newtheorem{theorem}{\em Theorem}

\newtheorem{lemma}{\em Lemma}

\newtheorem{corollary}{\em Corollary}
\newtheorem{observation}{\em Observation}

\journal{Latex Templates}

\begin{document}

\begin{frontmatter}

\title{Three results towards the approximation of special maximum matchings in graphs}

\author[label1]{Vahan Mkrtchyan}
\address[label1]{Department of Mathematics and Computer Science,\\ College of the Holy Cross, Worcester, MA, USA}

%
\ead{vahan.mkrtchyan@gssi.it}
%
%

\begin{abstract}
For a graph $G$ define the parameters $\ell(G)$ and $L(G)$ as the minimum and maximum value of $\nu(G\backslash F)$, where $F$ is a maximum matching of $G$ and $\nu(G)$ is the matching number of $G$. In this paper, we show that there is a small constant $c>0$, such that the following decision problem is NP-complete: given a graph $G$ and $k\leq \frac{|V|}{2}$, check whether there is a maximum matching $F$ in $G$, such that $|\nu(G\backslash F)-k|\leq c\cdot |V|$. Note that when $c=1$, this problem is polynomial time solvable as we observe in the paper. Since in any graph $G$, we have $L(G)\leq 2\ell(G)$, any polynomial time algorithm constructing a maximum matching of a graph is a 2-approximation algorithm for $\ell(G)$ and $\frac{1}{2}$-approximation algorithm for $L(G)$. We complement these observations by presenting two inapproximability results for $\ell(G)$ and $L(G)$.
\end{abstract}


\end{frontmatter}



\section{Introduction}
\label{IntroSection}

In this paper, we consider finite, undirected graphs without loops or parallel edges. The set of vertices and edges of a graph $G$ is denoted by $V(G)$ and $E(G)$, respectively. The degree of a vertex $v$ of $G$ is denoted by $d_{G}(v)$. Let $\Delta(G)$ and $\delta(G)$ be the maximum and minimum degree of a vertex of $G$. A graph is cubic if $\delta(G)=\Delta(G)=3$. A graph is bipartite if its set of vertices can be divided into two disjoint sets $V_1$ and $V_2$, such that every edge connects a vertex in $V_1$ to one in $V_2$. A graph is called a forest if it does not contain a cycle. Note that a forest can be disconnected. In a special case when it is connected, the graph is called a tree.

A matching in a graph $G$ is a subset of edges such that no vertex of $G$ is incident to two edges from the subset. A maximum matching is a matching that contains maximum possible number of edges. A maximum matching is perfect if every vertex of the graph is incident to an edge from the perfect matching. 

If $k\geq 0$, then a graph $G$ is called {$k$-edge colorable}, if its edges can be assigned colors from a set of $k$ colors so that adjacent
edges receive different colors. The smallest integer $k$, such that $G$ is $k$-edge-colorable is called chromatic index of $G$ and is denoted by $\chi'(G)$. The classical theorem of Shannon states that for any multi-graph $G$, $\Delta(G)\leq \chi'(G) \leq \left \lfloor \frac{3\Delta(G)}{2} \right \rfloor$ \cite{Shannon:1949,stiebitz:2012}. On the other hand, the classical theorem of Vizing states that for any multi-graph $G$, $\Delta(G)\leq \chi'(G) \leq \Delta(G)+\mu(G)$ \cite{stiebitz:2012,vizing:1964}. Here $\mu(G)$ is the maximum multiplicity of an edge of $G$. Note that when $G$ is a (simple) graph, this theorem implies that 
\[\Delta(G)\leq \chi'(G) \leq \Delta(G)+1.\]
So in particular, when $G$ is a cubic graph, we have that $3\leq \chi'(G) \leq 4$. The problem of deciding whether a cubic graph $G$ has $\chi'(G)=3$ is an NP-complete problem as demonstrated by Holyer \cite{Holyer}.



If $k<\chi'(G)$, we cannot color all edges of $G$ with $k$ colors. Thus it is reasonable to investigate the maximum number of edges that one can color with $k$ colors. A subgraph $H$ of a graph $G$ is called {maximum $k$-edge-colorable}, if $H$ is $k$-edge-colorable and contains maximum number of edges among all $k$-edge-colorable subgraphs. For $k\geq 0$ and a graph $G$, let
\[\nu_{k}(G) = \max \{ |E(H)| : H \text{ is a $k$-edge-colorable subgraph of } G \}. \]
Clearly, a $k$-edge-colorable subgraph is maximum if it contains exactly $\nu_k(G)$ edges. Note that $\nu_1(G)$ is the size of the maximum matching in $G$. Usually, we will shorten the notation $\nu_1(G)$ to $\nu(G)$.

It may seem that if we have a maximum $k$-edge-colorable subgraph of a graph, then by adding some edges to it, we can get a maximum $(k+1)$-edge-colorable subgraph. The example from Figure \ref{fig:Examplenu1nu2} shows that this is not true. It has a unique perfect matching, a matching covering all vertices of the graph, which contains the edge joining the two degree three vertices. However, the unique maximum $2$-edge-colorable subgraph of it contains all its eight edges except the edge joining the two degree three vertices.

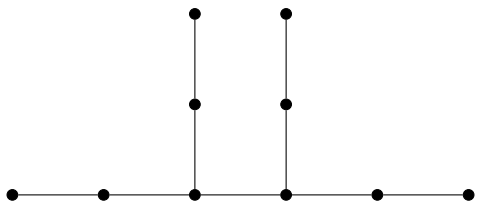
\begin{figure}[ht]
\centering
  
  \begin{center}

		\begin{tikzpicture}[scale = 0.6]
			
			
			
			

			\tikzstyle{every node}=[circle, draw, fill=black!50,inner sep=0pt, minimum width=4pt]
																							
			\node[circle,fill=black,draw] at (0,0) (n00) {};
			
			\node[circle,fill=black,draw] at (-2, 0) (nm20) {};
																								
			\node[circle,fill=black,draw] at (-4,0) (nm40) {};

                \node[circle,fill=black,draw] at (0,2) (n02) {};

                \node[circle,fill=black,draw] at (0,4) (n04) {};
                

                \node[circle,fill=black,draw] at (2,0) (n20) {};

                \node[circle,fill=black,draw] at (4,0) (n40) {};

                \node[circle,fill=black,draw] at (6,0) (n60) {};

                \node[circle,fill=black,draw] at (2,2) (n22) {};

                \node[circle,fill=black,draw] at (2,4) (n24) {};

			


			\path[every node]
			
			(n00) edge (nm20)
                (nm20) edge (nm40)

                (n00) edge (n02)
                (n02) edge (n04)

                (n00) edge (n20)
   

                (n20) edge (n40)
                (n40) edge (n60)
			
			(n20) edge (n22)
                (n22) edge (n24)

			;
		\end{tikzpicture}
																
	\end{center}
								
	\caption{A graph in which the maximum matching is not a subset of a maximum $2$-edge-colorable subgraph.}
	\label{fig:Examplenu1nu2}
\end{figure} 

In this paper, we deal with the following problem. Let $G$ be a graph. Define:
\[\ell(G)=\min\{\nu(G\backslash F): F\text{ is a maximum matching of }G\},\]
and
\[L(G)=\max\{\nu(G\backslash F): F\text{ is a maximum matching of }G\}.\]

Note that if $G$ is the path of length four (Figure \ref{fig:Example4path}), then $\ell(G)=1$ and $L(G)=2$. In the Master thesis of the author \cite{KM2008}, the problem of computing $\ell(G)$ and $L(G)$ is considered. The main contributions of the paper are that both of these parameters are NP-hard to compute in the connected, bipartite graphs of maximum degree three. On the positive side, \cite{KM2007} shows that both of the parameters are polynomial time computable when the input graph is a tree.

\cite{KM2008Artur} shows that in any graph $G$, $L(G)\leq 2\cdot \ell(G)$. Moreover, $L(G)\leq \frac{3}{2}\cdot \ell(G)$ provided that the graph $G$ contains a perfect matching \cite{KM2008Artur}. Finally, in \cite{KM2008Artur} the graphs with $L(G)= 2\ell(G)$ are characterized. This characterization implies that there is a polynomial time algorithm to test a given graph $G$ for $L(G)= 2\ell(G)$. Note that the situation is completely different for graphs containing a perfect matching. \cite{KM2008Artur} shows that the problem of testing a given graph $G$ with a perfect matching for $L(G)= \frac{3}{2}\ell(G)$ is NP-complete already in bridgeless cubic graphs. Note that these graphs always have a perfect matching by the classical theorem of Petersen \cite{Lovasz}.

\begin{figure}[ht]
\centering
  
  \begin{center}

		\begin{tikzpicture}[scale = 0.6]
			
			
			
			

			\tikzstyle{every node}=[circle, draw, fill=black!50,inner sep=0pt, minimum width=4pt]
																							
			\node[circle,fill=black,draw] at (0,0) (n00) {};

                \node[circle,fill=black,draw] at (2,0) (n20) {};

                \node[circle,fill=black,draw] at (4,0) (n40) {};

                \node[circle,fill=black,draw] at (6,0) (n60) {};

                \node[circle,fill=black,draw] at (8,0) (n80) {};



			\path[every node]
			
			(n00) edge (n20)
                (n20) edge (n40)
                (n40) edge (n60)
                (n60) edge (n80)

			;
		\end{tikzpicture}
																
	\end{center}
								
	\caption{$\ell(G)=1$ and $L(G)=2$ when $G$ is the path of length four.}
	\label{fig:Example4path}
\end{figure}
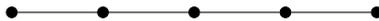

Since $L(G)$ is NP-hard to compute in the class of bipartite graphs, one may try to obtain some bounds for it. The main contribution of \cite{KM2005} is obtaining an upper bound for $L(G)$ in the class of bipartite graphs. The idea there is that in bipartite graphs $G$ one can compute $\nu_2(G)$ in polynomial time. This result is derived via maximum flows in networks. One can find its description in \cite{MK2024}, too (see Observation 2 there). Let us note that \cite{MK2006} shows that $\nu_2(G)=\nu(G)+L(G)$ in the class of trees in which the distance between any two pendant vertices is even. Recall that these trees are exactly the block-cut point graphs of arbitrary graphs \cite{Harary}. Note that this result does not hold for any tree as the example from Figure \ref{fig:Examplenu1nu2} demonstrates. In it $\nu_2(G)=8$, $\nu(G)=5$ and $L(G)=2$.

The problem of construction matchings with constraints is not new. See papers \cite{Cameron1,Cameron2,GHHL2005,Golumbic,IRT1978,Monnot} for other results that deal with the problem of construction matchings with restrictions. For the notions, facts and concepts that are not explained in the paper the reader is referred to the graph theory monograph \cite{west:1996}.

\section{Main results}
\label{MainSection}

In this section, we obtain the main results of the paper. We start with the additive approximation of our problem. More precisely, we will consider the following decision problem: 
\begin{problem}
    \label{prob:main} For a given graph $G$ and an integer $k$ with $k\leq \frac{|V|}{2}$, does there exist a maximum matching $F$, such that
\[|\nu(G\backslash F)-k|\leq f(|V|)?\]
\end{problem} Here $f$ is a fixed polynomial time computable function. Our first observation states:
\begin{observation}
    \label{obs:|V|=1} If $f(x)=x$, then Problem \ref{prob:main} can be solved in polynomial time.
\end{observation}
\begin{proof} Note that for any graph $G$ and its maximum matching $F$
\[\nu(G\backslash F)\leq \nu(G)\leq \frac{|V|}{2}.\]
Since $k\leq \frac{|V|}{2}$, we get:
\[|\nu(G\backslash F)-k|\leq \nu(G\backslash F)+k\leq \frac{|V|}{2}+\frac{|V|}{2}=|V|=f(|V|). \]
Thus, all instances of this problem have a trivial ``yes"-answer. The proof is complete. 
\end{proof}

Now, we are going to complement this observation with the following negative result. Its proof requires the following result from \cite{Hastad2001} whose proof relies on the classical PCP theorem \cite{CompComplexityAroraBarak2009}.

\begin{lemma}
    \label{lem:3SATinapprox} (Theorem 6.5 from \cite{Hastad2001}) Suppose $X=\{x_1,...,x_n\}$ is a set of boolean variables and $C=\{C_1,...,C_m\}$ is a set of clauses each of which contains exactly three literals of variables from $X$. Suppose that our instances are such that either all $m$ clauses are satisfiable, or at most $\left(\frac{7}{8}+\varepsilon\right)\cdot m$ clauses are satisfiable by any truth assignment. Here $\varepsilon\in \left(0,\frac{1}{8}\right)$ is any number. The goal is to decide which of these two cases hold for a given instance. This problem is NP-hard.
\end{lemma}

\begin{theorem}
    \label{thm:AdditiveApprox} If $f(x)=c\cdot x$, where $c\in \left(0, \frac{1}{256}\right)$, then Problem \ref{prob:main} is NP-complete in the class of connected bipartite graphs.
\end{theorem}

\begin{proof} Note that our problem is in NP. This just follows from the polynomial solvability of the maximum matching problem \cite{Lovasz}, and the polynomial computability of the linear function $f(x)=c\cdot x$. Now, we are going to reduce the problem from Lemma \ref{lem:3SATinapprox} to our problem. We will use some ideas from the Master thesis of the author \cite{KM2008}.

\begin{figure}[ht]
\centering
\begin{minipage}[b]{.5\textwidth}
  \begin{center}
	\begin{tikzpicture}[scale = 0.6]
			
			
			
			

                \draw[->, thick] (-1, 0) -- (10,0);
                \draw[->, thick] (0, -1) -- (0, 10);

                \node at (5,-1) {$4i-1$};
                \node at (7,-1) {$4i$};

                \node at (-1,3) {$4j-3$};
                \node at (-1,5) {$4j-2$};
                \node at (-1,7) {$4j-1$};
                \node at (-1,9) {$4j$};

                \node at (5.65,3) {$u_{11}$};
                \node at (7.65,3) {$u_{21}$};
                \node at (5.65,5) {$u_{12}$};
                \node at (7.65,5) {$u_{22}$};

                \node at (7.65,7) {$v_{22}$};
                \node at (7.65,9) {$v_{12}$};
                \node at (5.65,6.65) {$v_{21}$};
                \node at (5.65,8.65) {$v_{11}$};

                \draw [dashed] (5,3) -- (5,0);
                \draw [dashed] (7,3) -- (7,0);

                \draw [dashed] (5,3) -- (0,3);
                \draw [dashed] (5,5) -- (0,5);
                \draw [dashed] (5,7) -- (0,7);
                \draw [dashed] (5,9) -- (0,9);

			\tikzstyle{every node}=[circle, draw, fill=black!50,inner sep=0pt, minimum width=4pt]
																							
			\node[circle,fill=black,draw] at (5,3) (u11) {};
                \node[circle,fill=black,draw] at (5,5) (u12) {};
                \node[circle,fill=black,draw] at (7,3) (u21) {};
                \node[circle,fill=black,draw] at (7,5) (u22) {};

                \node[circle,fill=black,draw] at (5,9) (v11) {};
                \node[circle,fill=black,draw] at (7,9) (v12) {};
                \node[circle,fill=black,draw] at (5,7) (v21) {};
                \node[circle,fill=black,draw] at (7,7) (v22) {};



			\path[every node]
			
			(u11) edge (u12)
                (u21) edge (u22)
                (u12) edge (v21)
                (u22) edge (v22)

                (v21) edge (v22)
                (v22) edge (v12)
                (v11) edge (v12)

			;
		\end{tikzpicture}

	\end{center}
	
	\caption{The gadget corresponding to the variable\\ $x_i$ and the clause $C_j$.}\label{fig:Gadgeta}
\end{minipage}%
\begin{minipage}[b]{.5\textwidth}
  	\begin{center}
	\centering
  
  \begin{center}

		\begin{tikzpicture}[scale = 0.6]
			
			
			
			

                \draw[->, thick] (-1, 0) -- (10,0);
                \draw[->, thick] (0, -1) -- (0, 10);

                \node at (2.65,-1) {$4i-3$};
                \node at (5,-1) {$4i-2$};
                 \node at (7.35,-1) {$4i-1$};
                \node at (9.35,-1) {$4i$};

                \node at (-1,5) {$4j-1$};
                \node at (-1,7) {$4j$};

                \node at (3,5.35) {$u_{11}$};
                \node at (3, 7.35) {$u_{21}$};
                \node at (5,5.35) {$u_{12}$};
                \node at (5,7.35) {$u_{22}$};

                \node at (10,5.35) {$v_{22}$};
                \node at (9.65,7.35) {$v_{12}$};
                \node at (7.65,5.35) {$v_{21}$};
                \node at (7.65,7.35) {$v_{11}$};

                \draw [dashed] (2.65,5) -- (2.65,0);
                \draw [dashed] (5,5) -- (5,0);
                \draw [dashed] (7.35,5) -- (7.35,0);
                \draw [dashed] (9.35,5) -- (9.35,0);

                 \draw [dashed] (2.65,5) -- (0,5);
                 \draw [dashed] (2.65,7) -- (0,7);

			\tikzstyle{every node}=[circle, draw, fill=black!50,inner sep=0pt, minimum width=4pt]
																							
			\node[circle,fill=black,draw] at (2.65,5) (u11) {};
                \node[circle,fill=black,draw] at (5,5) (u12) {};
                \node[circle,fill=black,draw] at (2.65,7) (u21) {};
                \node[circle,fill=black,draw] at (5,7) (u22) {};

                \node[circle,fill=black,draw] at (7.35,7) (v11) {};
                \node[circle,fill=black,draw] at (9.35,7) (v12) {};
                \node[circle,fill=black,draw] at (7.35,5) (v21) {};
                \node[circle,fill=black,draw] at (9.35,5) (v22) {};



			\path[every node]
			
			(u11) edge (u12)
                (u12) edge (v21)
                (v21) edge (v22)
                (v22) edge (v12)
                (v11) edge (v12)
                (v11) edge (u22)
                (u21) edge (u22)

			;
		\end{tikzpicture}
																
	\end{center}
								
	\caption{The gadget corresponding to the literal $\overline{x}_i$ and clause $C_j$.}
	\label{fig:Gadgetb}
	\end{center}
\end{minipage}
\end{figure}

Let $I=(X, C)$ be an instance of the first problem. Now we are going to describe a graph $G_I$ corresponding to it. Our graph is going to have vertices which will be integral points on the plane. In other words, our vertices will be pairs $(x,y)$ where both $x$ and $y$ are integers. 

Suppose $x_i$ ($1\leq i\leq n$) is a boolean variable from $X$ appearing in a clause $C_j$ ($1\leq j\leq m$) from $C$. If $x_i$ appears as a variable in $C_j$, then the graph corresponding to it is from Figure \ref{fig:Gadgeta}, and if $x_i$ appears as a negated variable in $C_j$, then the graph corresponding to it is from Figure \ref{fig:Gadgetb}. This graph is going to appear as a part of a larger graph. Sometimes, we will prefer not to draw the vertices $u_{ij}$ ($1\leq i,j \leq 2$). Thus, we will use a conventional sign for them. This sign is the one from Figure \ref{fig:ConventionalSign}. The vertices $v_{ij}$ ($1\leq i, j\leq 2$) are the ones from the corresponding graph. 

\begin{figure}[ht]
\centering
  
  \begin{center}

		\begin{tikzpicture}[scale = 0.6]
			
			
			
			

                \node at (0,0.5) {$v_{11}$};
                \node at (2,0.5) {$v_{12}$};
                \node at (0,-1.5) {$v_{21}$};
                \node at (2.5,-1.65) {$v_{22}$};

			\tikzstyle{every node}=[circle, draw, fill=black!50,inner sep=0pt, minimum width=4pt]
																							
			\node[circle,fill=black,draw] at (0,0) (v11) {};
                \node[circle,fill=black,draw] at (2,0) (v12) {};
                \node[circle,fill=black,draw] at (2,-2) (v22) {};
                \node[circle,fill=black,draw] at (0,-2) (v21) {};



			\path[every node]
			
			(v11) edge (v12)
                (v12) edge (v22)
                (v22) edge (v21)

			;
		\end{tikzpicture}
																
	\end{center}
								
	\caption{The conventional sign.}
	\label{fig:ConventionalSign}
\end{figure}
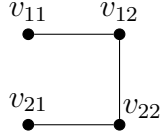

Now, if $C_j=t_{j_1}\vee t_{j_2}\vee t_{j_3}$  ($1\leq j_1 < j_2<j_3\leq m$) is a clause containing literals $t_{j_1}$, $t_{j_2}$ and $t_{j_3}$ of variables $x_{j_1}$, $x_{j_2}$ $x_{j_3}$, respectively, then the graph $G(C_j)$ corresponding to $C_j$ is the one from Figure \ref{fig:GraphCorrespondingClause}.

\begin{figure}[ht]
\centering
  
  \begin{center}

		\begin{tikzpicture}[scale = 0.6]
			
			
			
			


                \draw[->, thick] (-1, 0) -- (17,0);
                \draw[->, thick] (0, -1) -- (0, 10);

                \node at (2,-1) {$4j_1-1$};
                \node at (4,-1) {$4j_1$};

                \node at (7,-1) {$4j_2-1$};
                \node at (9,-1) {$4j_2$};

                 \node at (12,-1) {$4j_3-1$};
                \node at (14,-1) {$4j_3$};

                \node at (-1.5,7) {$4j-1$};
                \node at (-1,9) {$4j$};
                 \node at (-1.5,5) {$4j-2$};
                \node at (-1.5,3) {$4j-3$};


                \draw [dashed] (2,7) -- (2,0);
                \draw [dashed] (4,7) -- (4,0);
                \draw [dashed] (7,7) -- (7,0);
                \draw [dashed] (9,7) -- (9,0);
                \draw [dashed] (12,7) -- (12,0);
                \draw [dashed] (14,7) -- (14,0);

                \draw [dashed] (2,7) -- (0,7);
                \draw [dashed] (2,9) -- (0,9);

			\tikzstyle{every node}=[circle, draw, fill=black!50,inner sep=0pt, minimum width=4pt]
																							
			\node[circle,fill=black,draw] at (2,7) (v27) {};
                \node[circle,fill=black,draw] at (4,7) (v47) {};
                \node[circle,fill=black,draw] at (2,9) (v29) {};
                \node[circle,fill=black,draw] at (4,9) (v49) {};

                \node[circle,fill=black,draw] at (7,7) (v77) {};
                \node[circle,fill=black,draw] at (7,9) (v79) {};
                \node[circle,fill=black,draw] at (9,7) (v97) {};
                \node[circle,fill=black,draw] at (9,9) (v99) {};

                \node[circle,fill=black,draw] at (12,7) (v127) {};
                \node[circle,fill=black,draw] at (12,9) (v129) {};
                \node[circle,fill=black,draw] at (14,7) (v147) {};
                \node[circle,fill=black,draw] at (14,9) (v149) {};

                \node[circle,fill=black,draw] at (0,7) (v07) {};
                \node[circle,fill=black,draw] at (0,9) (v09) {};
                 \node[circle,fill=black,draw] at (0,5) (v05) {};
                \node[circle,fill=black,draw] at (0,3) (v03) {};



			\path[every node]
			
			(v27) edge (v47)
                (v47) edge (v49)
                (v49) edge (v29)

                (v77) edge (v97)
                (v97) edge (v99)
                (v99) edge (v79)

                (v127) edge (v147)
                (v147) edge (v149)
                (v149) edge (v129)

                (v07) edge (v09)

                (v09) edge [bend right] (v07)
                (v07) edge (v49)
                (v07) edge (v99)
                (v07) edge (v149)

                (v03) edge [bend left] (v05)
                (v03) edge (v49)
                (v03) edge (v99)
                (v03) edge (v149)

			;
		\end{tikzpicture}
																
	\end{center}
								
	\caption{The graph $G(C_j)$ corresponding to the clause $C_j$.}
	\label{fig:GraphCorrespondingClause}
\end{figure}
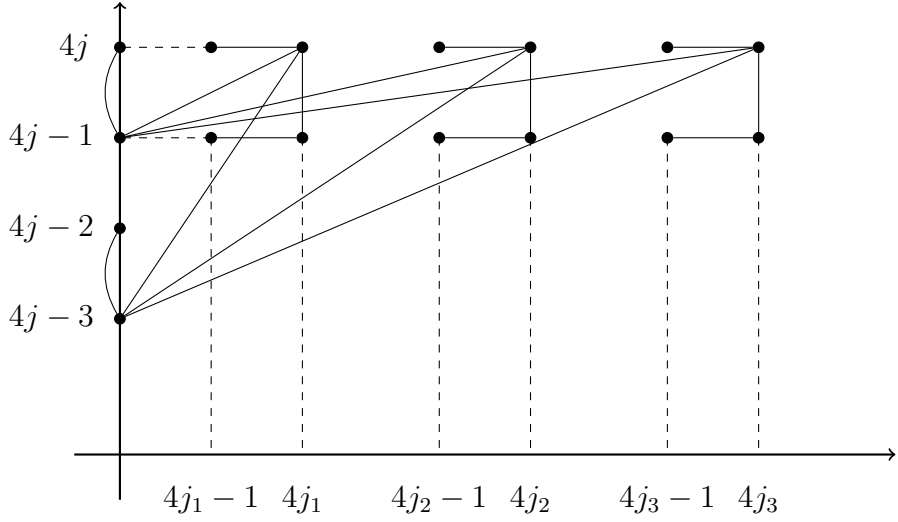

Now, for $i=1,...,n$ let $C_{j_1}, ..., C_{{j_{r(i)}}}$, where $r(i)\geq 1$, be the clauses of $C$ containing a literal of $x_i$. We can assume that $j_1<j_2<...<j_{r(i)}$. Define a graph $G(I)$ corresponding to $I$ as follows: if $G(C_1)$, ..., $G(C_m)$ are the graphs corresponding to clauses $C_1,..., C_m$, then for $i=1,...,n$ cyclically connect $G(C_{j_1}), ..., G(C_{j_{{r(i)}}})$ Figure \ref{fig:GraphCorrespondingVariable}.

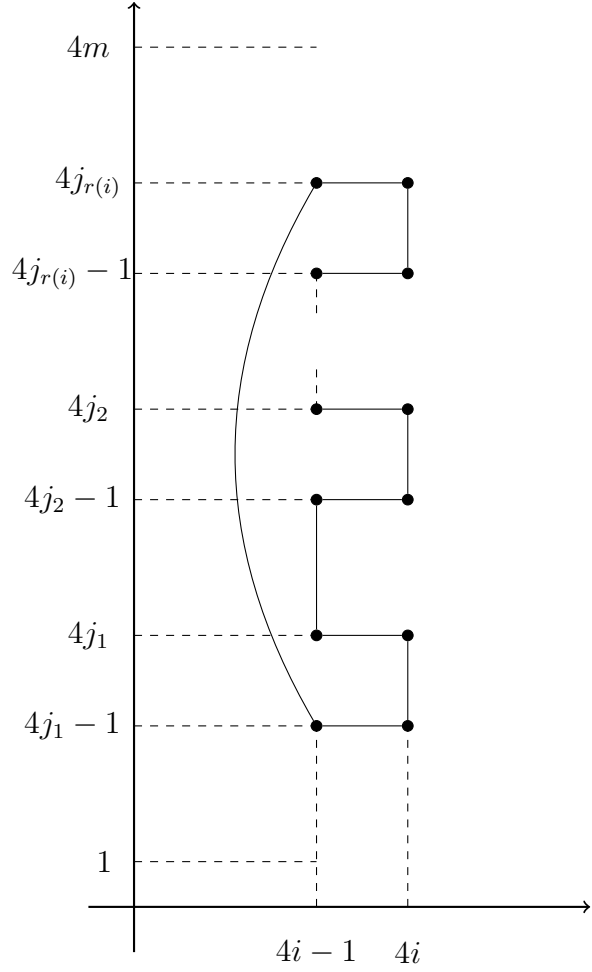
\begin{figure}[ht]
\centering
  
  \begin{center}

		\begin{tikzpicture}[scale = 0.6]
			
			
			
			


                \draw[->, thick] (-1, 0) -- (10,0);
                \draw[->, thick] (0, -1) -- (0,20);

                \node at (4,-1) {$4i-1$};
                \node at (6,-1) {$4i$};

                \node at (-0.65,1) {$1$};

                \node at (-1.35,4) {$4j_1-1$};
                \node at (-1,6) {$4j_1$};

                \node at (-1.35,9) {$4j_2-1$};
                \node at (-1,11) {$4j_2$};

                \node at (-1.35,14) {$4j_{r(i)}-1$};
                \node at (-1,16) {$4j_{r(i)}$};

                \node at (-1,19) {$4m$};


                \draw [dashed] (4,11) -- (4,12);
                \draw [dashed] (4,14) -- (4,13);

                \draw [dashed] (4,0) -- (4,4);
                \draw [dashed] (6,0) -- (6,4);

                \draw [dashed] (0,1) -- (4,1);
                \draw [dashed] (0,4) -- (4,4);
                \draw [dashed] (0,6) -- (4,6);
                \draw [dashed] (0,9) -- (4,9);
                \draw [dashed] (0,11) -- (4,11);
                \draw [dashed] (0,14) -- (4,14);
                \draw [dashed] (0,16) -- (4,16);
                \draw [dashed] (0,19) -- (4,19);

			\tikzstyle{every node}=[circle, draw, fill=black!50,inner sep=0pt, minimum width=4pt]
																							
			\node[circle,fill=black,draw] at (4,4) (v44) {};
                \node[circle,fill=black,draw] at (6,4) (v64) {};
                \node[circle,fill=black,draw] at (6,6) (v66) {};
                \node[circle,fill=black,draw] at (4,6) (v46) {};

                \node[circle,fill=black,draw] at (4,9) (v49) {};
                \node[circle,fill=black,draw] at (6,9) (v69) {};
                \node[circle,fill=black,draw] at (6,11) (v611) {};
                \node[circle,fill=black,draw] at (4,11) (v411) {};

                 \node[circle,fill=black,draw] at (4,14) (v414) {};
                \node[circle,fill=black,draw] at (6,14) (v614) {};
                \node[circle,fill=black,draw] at (6,16) (v616) {};
                \node[circle,fill=black,draw] at (4,16) (v416) {};



			\path[every node]
			
			(v44) edge (v64)
                (v64) edge (v66)
                (v66) edge (v46)

                (v46) edge (v49)

                (v49) edge (v69)
                (v69) edge (v611)
                (v611) edge (v411)

                (v414) edge (v614)
                (v614) edge (v616)
                (v616) edge (v416)

                (v44) edge [bend left] (v416)

			;
		\end{tikzpicture}
																
	\end{center}
								
	\caption{Cyclically joining the subgraphs that correspond to the variable $x_i$.}
	\label{fig:GraphCorrespondingVariable}
\end{figure}

\begin{figure}[ht]
\centering
  
  \begin{center}

		\begin{tikzpicture}[scale = 0.6]
			
			
			
			


               \node at (8,4) {Connecting to one};
               \node at (8,3) {of the three vertices};
               \node at (8,2) {$u_{11}$ of $G(C_1)$};

               \node at (8,9) {Connecting to one};
               \node at (8,8) {of the three vertices};
               \node at (8,7) {$u_{11}$ of $G(C_m)$};

                \draw[->, thick] (-2, 0) -- (10,0);
                \draw[->, thick] (0, -1) -- (0, 10);

                \node at (-1.35,-1) {$-1$};

                \node at (1,1) {$1$};
                \node at (1,2) {$2$};
                \node at (1,3) {$3$};
                \node at (1,4) {$4$};
                \node at (1.35,8) {$4m-1$};
                \node at (1,9) {$4m$};

                \draw [dashed] (-1,1) -- (0.5,1);
                \draw [dashed] (-1,2) -- (0.5,2);
                \draw [dashed] (-1,3) -- (0.5,3);
                \draw [dashed] (-1,4) -- (0.5,4);

                 \draw [dashed] (-1,8) -- (0.5,8);
                  \draw [dashed] (-1,9) -- (0.5,9);

                   \draw [dashed] (-1,1) -- (-1,0);

                    \draw [dashed] (-1,4) -- (-1,5);
                    \draw [dashed] (-1,8) -- (-1,7);



			\tikzstyle{every node}=[circle, draw, fill=black!50,inner sep=0pt, minimum width=4pt]
																							
			\node[circle,fill=black,draw] at (-1,1) (vm11) {};
                \node[circle,fill=black,draw] at (-1,2) (vm12) {};
                \node[circle,fill=black,draw] at (-1,3) (vm13) {};
                \node[circle,fill=black,draw] at (-1,4) (vm14) {};

                \node[circle,fill=black,draw] at (-1,8) (vm14minus1) {};
                \node[circle,fill=black,draw] at (-1,9) (vm14m) {};



			\path[every node]
			
			(vm11) edge (vm12)
                (vm12) edge (vm13)
                (vm13) edge (vm14)

                (vm14minus1) edge (vm14m)

                (vm14) edge [bend left] (5,4)
                (vm14m) edge [bend left] (5,9)

			;
		\end{tikzpicture}
																
	\end{center}
								
	\caption{Making sure that the resulting graph is connected.}
	\label{fig:GraphConnected}
\end{figure}
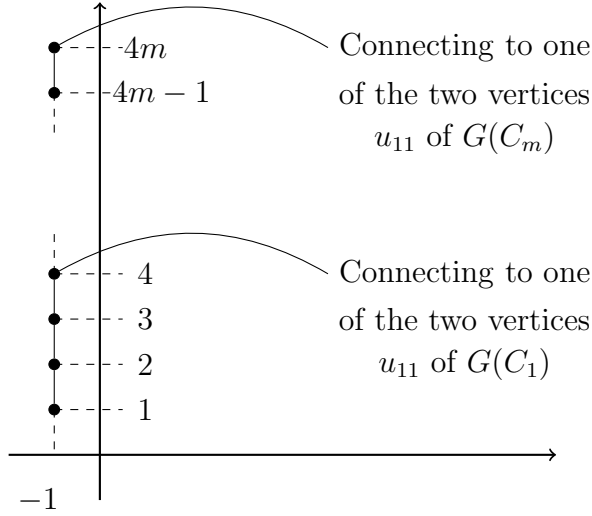

As it is stated in \cite{KM2008}, the constructed graph $G(I)$ may not be connected. Therefore in order to obtain a connected one let us consider a graph $G_I$ constructed from $G(I)$ as it is stated in Figure \ref{fig:GraphConnected}. Note that $G_I$ is a connected graph of maximum degree four with $|V(G_I)|=32m$, $|E(G_I)|=37m-1$, and $\nu(G_I)=\frac{|V(G_I)|}{2}=16m$. 

Let us show that $G_I$ is a bipartite graph. Define: $V_0$ to be the set of vertices $(x,y)$ of $G_I$ such that $x+y$ is an even number, and let $V_1$ to be the set of vertices $(x,y)$ of $G_I$ such that $x+y$ is an odd number. Note that $V_0$, $V_1$ form a partition of $V(G_I)$. Moreover, since every edge of $G_I$ joins a vertex from $V_0$ to one from $V_1$, we have that they form a bipartition. Thus, $G_I$ is a bipartite graph.

Consider an instance $(G_I,k)$ of Problem \ref{prob:main}, where $k=11m-1$. Note that $k\leq \frac{|V(G_I)|}{2}$. Moreover, $(G_I,k)$ can be constructed from $I=(X,C)$ in polynomial time. In order to demonstrate the NP-hardness of Problem \ref{prob:main}, it suffices to show that:
\begin{enumerate}
    \item [(A)] Suppose that all clauses of $C$ are satisfiable. Then there is a maximum matching $F$ of $G_I$, such that \[|\nu(G_I\backslash F)-k|\leq f(|V|).\]

    \item [(B)] Suppose that at most $\left(\frac{7}{8}+\varepsilon\right)\cdot m$ clauses of $C$ are satisfiable by any truth assignment. Then for any maximum matching $F$ of $G_I$, we have
    \[|\nu(G_I\backslash F)-k|> f(|V|).\]
\end{enumerate} Let us prove these statements one by one.

\medskip

(A) Let $\Tilde{\alpha}=(\alpha_1,...,\alpha_n)$ be a truth assignment satisfying all clauses of $C$. Consider a subset $F$ of edges of $G_I$ defined as follows:
\begin{itemize}
    \item add the unique perfect matching of the path $(-1,1)-(-1, 4m)$ in Figure \ref{fig:GraphConnected} to $F$,

    \item add all edges $(0, 4j-3) (0, 4j-2)$, $(0, 4j-1) (0, 4j)$ for $j = 1, ..., m$ to $F$,

    \item add all edges $u_{11}u_{12}$, $u_{21}u_{22}$ to $F$, for $i=1,...,n$,

    \item if $\alpha_i=TRUE$, then add the vertical edges of the cycle corresponding to the variable $x_i$ in Figure \ref{fig:GraphCorrespondingVariable}, and if $\alpha_i=FALSE$, then add the horizontal edges of the cycle corresponding to the variable $x_i$ in Figure \ref{fig:GraphCorrespondingVariable}.
\end{itemize} Note that $F$ is a matching in $G_I$ covering all its vertices. Hence, it is a perfect matching. Let us compute $\nu(G_I\backslash F)$. Note that on the path $(-1,1)-(-1, 4m)$ in Figure \ref{fig:GraphConnected}, we can take $2m-1$ edges. Moreover, since all clauses are satisfied, in each copy of $G(C_j)$ ($1\leq j\leq m$) we can take nine edges. Hence,
\[\nu(G_I\backslash F)=2m-1+9m=11m-1=k.\]
Thus, the inequality
\[|\nu(G_I\backslash F)-k|\leq f(|V|)\]
is trivially true.

\medskip

(B) Suppose that at most $\left(\frac{7}{8}+\varepsilon\right)\cdot m$ clauses of $C$ are satisfiable by any truth assignment. Let $F$ be any perfect matching of $G_I$. Note that since $F$ covers all vertices, we have that the edge $(-1,1)(-1,2)\in F$ (Figure \ref{fig:GraphConnected}). By a similar reasoning, we have that $F$ contains the unique perfect matching of the path $(-1,1)-(-1,4m)$ (Figure \ref{fig:GraphConnected}), for $j=1,...,m$ edges $(0,4j-3)(0,4j-2)$, $(0,4j-1)(0,4j)$ from Figure \ref{fig:GraphCorrespondingClause}. This implies that it must contain all edges $u_{11}u_{12}$, $u_{21}u_{22}$ (Figures \ref{fig:Gadgeta} and \ref{fig:Gadgetb}). This implies that for $i=1,..,n$ the remaining edges of $F$ induce a perfect matching in each of the subgraphs corresponding to the variable $x_i$ (Figure \ref{fig:GraphCorrespondingVariable}).

Define a truth assignment $\Tilde{\gamma}=(\gamma_1,...,\gamma_n)$ corresponding to $F$ as follows: if the vertical edges of the cycle corresponding to the variable $x_i$ (Figure \ref{fig:GraphCorrespondingVariable}) belong to $F$, then set $\gamma_i=TRUE$, and if the horizontal edges of the cycle corresponding to the variable $x_i$ (Figure \ref{fig:GraphCorrespondingVariable}) belong to $F$, then set $\gamma_i=FALSE$. Note that if $\Tilde{\gamma}$ satisfies the clause $C_j$ then in $G\backslash F$ we can take nine edges from it, and if $\Tilde{\gamma}$ does not satisfy the clause $C_j$ then in $G\backslash F$ we can take eight edges from it. Finally, we can always assume that in $G\backslash F$, $F$ takes the remaining edges of the path $(-1,1)-(-1,4m)$ (Figure \ref{fig:GraphConnected}). Thus,
\begin{equation}
    \label{eq:PerfectMatchingAssignment} \nu(G_I\backslash F)=2m-1+9\cdot |Sat(\Tilde{\gamma})|+8\cdot (m-|Sat(\Tilde{\gamma})|)=10m-1+|Sat(\Tilde{\gamma})|.
\end{equation} Here $Sat(\Tilde{\gamma})$ denotes the set of clauses of $C$ satisfied by $\Tilde{\gamma}$. Therefore,
\begin{align*}
    |\nu(G_I\backslash F)-k| &=\left|10m-1+|Sat(\Tilde{\gamma})|-11m+1\right|=m-|Sat(\Tilde{\gamma})|.
\end{align*} By our assumption
\[|Sat(\Tilde{\gamma})|\leq \left(\frac{7}{8}+\varepsilon\right)\cdot m.\]
Hence
\begin{equation*}
    |\nu(G_I\backslash F)-k| =m-|Sat(\Tilde{\gamma})|\geq m-\left(\frac{7}{8}+\varepsilon\right)\cdot m>c\cdot 32m=c\cdot |V(G_I)|,
\end{equation*} since
\[c\cdot 32<1-\frac{7}{8}-\varepsilon,\]
or
\begin{equation}
    \label{eq:final} c<\frac{1}{256}-\frac{\varepsilon}{32}.
\end{equation} Thus, if for a given $c\in \left(0, \frac{1}{256}\right)$ we choose $\varepsilon\in \left(0, \frac{1}{8}\right)$ so that inequality (\ref{eq:final}) holds, then we will get the result. The proof is complete.
\end{proof}

\begin{corollary}
    \label{cor:Sublinear} Problem \ref{prob:main} is NP-hard for every sublinear and polynomial time computable function $f$. A function $f: \mathbb{N} \rightarrow \mathbb{N}$ on positive integers is called sublinear, if 
\[\lim_{n\rightarrow +\infty}\frac{f(n)}{n}=0.\]
\end{corollary} Examples of such functions are $f(x)=C$, where $C$ is a constant, $f(x)=\log x$, $f(x)=\sqrt{x}$, etc..

\medskip

Now we turn to the multiplicative approximation of our parameters. As we said before, \cite{KM2008Artur} shows that in any graph $G$, $L(G)\leq 2\cdot \ell(G)$. This means that any algorithm (for example, those presented in \cite{Edmonds1,Edmonds2,EvenKariv,Lovasz,MicaliVazirani}) for finding a maximum matching in a given graph $G$ is a 2-approximation algorithm for $\ell(G)$, and a $\frac{1}{2}$-approximation algorithm for $L(G)$. Recall that by the standards of the area of approximation algorithms \cite{Vazirani}, a polynomial time algorithm constructing some maximum matching $F$ in a given graph $G$ is a $(1+\varepsilon)$-approximation algorithm for $\ell(G)$, if for every input graph $G$
\[\nu(G\backslash F)\leq (1+\varepsilon)\cdot \ell(G).\]
Similarly, it is a $(1-\varepsilon)$-approximation algorithm for $L(G)$, if for every input graph $G$
\[\nu(G\backslash F)\geq (1-\varepsilon)\cdot L(G).\]
We are ready to present our next result.

\begin{theorem}
    \label{thm:L(G)inapprox} If $P\neq NP$, then for all $\varepsilon\in \left(0, \frac{1}{88}\right)$ there is no polynomial time $\left(1-\varepsilon\right)$-approximation algorithm for $L(G)$ in the class of connected bipartite graphs $G$.
\end{theorem}

\begin{proof} Suppose that $\mathcal{A}$ is a polynomial time $\left(1-\varepsilon\right)$-approximation algorithm for $L(G)$ in the class of all graphs $G$. Let $F$ be the maximum matching that it returns. This means that 
\[\nu(G\backslash F)\geq (1-\varepsilon)\cdot L(G)\]
for every input graph $G$. Define:
\[\delta=11(1-\varepsilon)-10-\frac{7}{8}\]
Note that for all $\varepsilon\in \left(0, \frac{1}{88}\right)$ we have $\delta\in  \left(0, \frac{1}{8}\right)$ and
\begin{equation}
    \label{eq:epsilondelta} 10+ \frac{7}{8}+\delta=11(1-\varepsilon).
\end{equation}

For every instance $I=(X, C)$ of Lemma \ref{lem:3SATinapprox} consider the graph $G_I$ constructed in Theorem \ref{thm:AdditiveApprox} with respect to $\delta$ defined above. We consider two cases.

\medskip

Case (I): There is a truth assignment satisfying all clauses $C$. In this case we have $L(G_I)\geq 11m-1$ as we saw in the proof of case (A) of Theorem \ref{thm:AdditiveApprox}. Hence, for the maximum matching $F$ returned by $\mathcal{A}$ we will have
\[\nu(G_I\backslash F)\geq (1-\varepsilon)\cdot L(G_I)\geq (1-\varepsilon)\cdot (11m-1). \]

\medskip

Case (II): Suppose that our instance $I$ is such that at most $\left(\frac{7}{8}+\delta\right)\cdot m$ clauses of $C$ are satisfiable by any truth assignment. Then, as we saw in equation (\ref{eq:PerfectMatchingAssignment}) from the proof of case (B) of Theorem \ref{thm:AdditiveApprox}, 
\[L(G_I)\leq 10m-1+ \left(\frac{7}{8}+\delta\right)\cdot m. \]
Hence, for the maximum matching $F$ returned by $\mathcal{A}$, we will have
\[\nu(G_I\backslash F)\leq L(G_I)\leq 10m-1+ \left(\frac{7}{8}+\delta\right)\cdot m.\]
Thus, using equality (\ref{eq:epsilondelta}), we get:
\begin{align*}
   \nu(G_I\backslash F) &\leq  10m-1+ \left(\frac{7}{8}+\delta\right)\cdot m=\left(10+\frac{7}{8}+\delta\right)\cdot m-1\\
                        &=11(1-\varepsilon)\cdot m-1< 11(1-\varepsilon)\cdot m-1+\varepsilon=(1-\varepsilon)\cdot (11m-1).
\end{align*}
 Thus, using our algorithm $\mathcal{A}$ we can solve NP-hard instances from Lemma \ref{lem:3SATinapprox} by simply comparing $\nu(G_I\backslash F)$ with $(1-\varepsilon)\cdot (11m-1)$. Hence $P=NP$. The proof is complete.
\end{proof}

\begin{theorem}
    \label{thm:l(G)inapprox} If $P\neq NP$, then for all $\varepsilon\in \left(0, \frac{1}{80}\right)$ there is no polynomial time $\left(1+\varepsilon\right)$-approximation algorithm for $\ell(G)$ in the class of connected bipartite graphs $G$.
\end{theorem}

\begin{proof} Suppose that $\mathcal{B}$ is a polynomial time $\left(1+\varepsilon\right)$-approximation algorithm for $\ell(G)$ in the class of all graphs $G$. Let $F_0$ be the maximum matching that it returns. This means that 
\[\nu(G\backslash F_0)\leq (1+\varepsilon)\cdot \ell(G)\]
for every input graph $G$.  Define:
\[\delta=11-\frac{7}{8}-10(1+\varepsilon).\]
Note that for all $\varepsilon\in \left(0, \frac{1}{80}\right)$, we have $\delta\in  \left(0, \frac{1}{8}\right)$ and
\begin{equation}
    \label{eq:epsilondeltaPrime} 11- \frac{7}{8}-\delta=10(1+\varepsilon).
\end{equation}

For every instance $I=(X, C)$ of the problem from Lemma \ref{lem:3SATinapprox} with $\delta\in  \left(0, \frac{1}{8}\right)$, we are going to construct a graph $G_I$ like we did in Theorem \ref{thm:AdditiveApprox}. Let $I=(X, C)$ be an instance of the first problem. As before, our graph is going to have vertices which will be integral points on the plane. In other words, our vertices will be pairs $(x,y)$ where both $x$ and $y$ are integers. 

Suppose $x_i$ ($1\leq i\leq n$) is a boolean variable from $X$ appearing in a clause $C_j$ ($1\leq j\leq m$) from $C$. If $x_i$ appears as a variable in $C_j$, then the graph corresponding to it is from Figure \ref{fig:Gadgeta}, and if $x_i$ appears as a negated variable in $C_j$, then the graph corresponding to it is from Figure \ref{fig:Gadgetb}. This graph is going to appear as a part of a larger graph. We will prefer not to draw the vertices $u_{ij}$ ($1\leq i,j \leq 2$). Thus, we will use the same conventional sign for them Figure \ref{fig:ConventionalSign}. The vertices $v_{ij}$ ($1\leq i, j\leq 2$) are the ones from the corresponding graph. 

Now, if $C_j=t_{j_1}\vee t_{j_2}\vee t_{j_3}$  ($1\leq j_1 < j_2<j_3\leq m$) is a clause containing literals $t_{j_1}$, $t_{j_2}$ and $t_{j_3}$ of variables $x_{j_1}$, $x_{j_2}$ $x_{j_3}$, respectively, then the graph $G(C_j)$ corresponding to $C_j$ is the one from Figure \ref{fig:GraphCorrespondingClausePrime}.

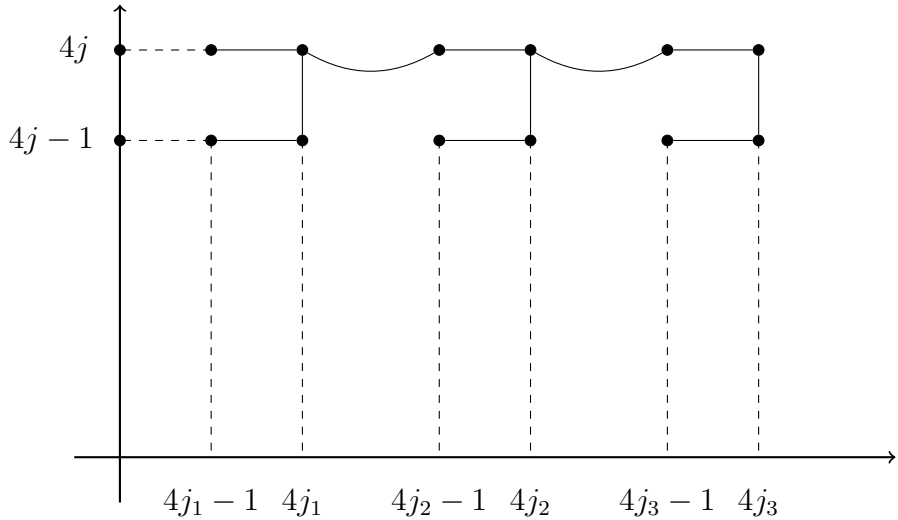
\begin{figure}[ht]
\centering
  
  \begin{center}

		\begin{tikzpicture}[scale = 0.6]
			
			
			
			


                \draw[->, thick] (-1, 0) -- (17,0);
                \draw[->, thick] (0, -1) -- (0, 10);

                \node at (2,-1) {$4j_1-1$};
                \node at (4,-1) {$4j_1$};

                \node at (7,-1) {$4j_2-1$};
                \node at (9,-1) {$4j_2$};

                 \node at (12,-1) {$4j_3-1$};
                \node at (14,-1) {$4j_3$};

                \node at (-1.5,7) {$4j-1$};
                \node at (-1,9) {$4j$};


                \draw [dashed] (2,7) -- (2,0);
                \draw [dashed] (4,7) -- (4,0);
                \draw [dashed] (7,7) -- (7,0);
                \draw [dashed] (9,7) -- (9,0);
                \draw [dashed] (12,7) -- (12,0);
                \draw [dashed] (14,7) -- (14,0);

                \draw [dashed] (2,7) -- (0,7);
                \draw [dashed] (2,9) -- (0,9);

			\tikzstyle{every node}=[circle, draw, fill=black!50,inner sep=0pt, minimum width=4pt]
																							
			\node[circle,fill=black,draw] at (2,7) (v27) {};
                \node[circle,fill=black,draw] at (4,7) (v47) {};
                \node[circle,fill=black,draw] at (2,9) (v29) {};
                \node[circle,fill=black,draw] at (4,9) (v49) {};

                \node[circle,fill=black,draw] at (7,7) (v77) {};
                \node[circle,fill=black,draw] at (7,9) (v79) {};
                \node[circle,fill=black,draw] at (9,7) (v97) {};
                \node[circle,fill=black,draw] at (9,9) (v99) {};

                \node[circle,fill=black,draw] at (12,7) (v127) {};
                \node[circle,fill=black,draw] at (12,9) (v129) {};
                \node[circle,fill=black,draw] at (14,7) (v147) {};
                \node[circle,fill=black,draw] at (14,9) (v149) {};




			\path[every node]
			
			(v27) edge (v47)
                (v47) edge (v49)
                (v49) edge (v29)

                (v77) edge (v97)
                (v97) edge (v99)
                (v99) edge (v79)

                (v127) edge (v147)
                (v147) edge (v149)
                (v149) edge (v129)

                (v07) edge (v09)



                (v49) edge [bend right] (v79)
                (v99) edge [bend right] (v129)

			;
		\end{tikzpicture}
																
	\end{center}
								
	\caption{The graph $G(C_j)$ corresponding to the clause $C_j$.}
	\label{fig:GraphCorrespondingClausePrime}
\end{figure}

Now, for $i=1,...,n$ let $C_{j_1}, ..., C_{{j_{r(i)}}}$, where $r(i)\geq 1$, be the clauses of $C$ containing a literal of $x_i$. We can assume that $j_1<j_2<...<j_{r(i)}$. Define a graph $G(I)$ corresponding to $I$ as follows: if $G(C_1)$, ..., $G(C_m)$ are the graphs corresponding to clauses $C_1,..., C_m$, then for $i=1,...,n$ cyclically connect $G(C_{j_1}), ..., G(C_{j_{{r(i)}}})$ Figure \ref{fig:GraphCorrespondingVariable}.

As before, the constructed graph $G(I)$ may not be connected. Therefore, in order to obtain a connected one, let us consider a graph $G_I$ constructed from $G(I)$ as it is stated in Figure \ref{fig:GraphConnected}. Note that $G_I$ is a connected graph of maximum degree three with $|V(G_I)|=28m$, $|E(G_I)|=31m-1$, and $\nu(G_I)=\frac{|V(G_I)|}{2}=14m$. 

Let us show that $G_I$ is a bipartite graph. Define: $V_0$ to be the set of vertices $(x,y)$ of $G_I$ such that $x+y$ is an even number, and let $V_1$ to be the set of vertices $(x,y)$ of $G_I$ such that $x+y$ is an odd number. Note that $V_0$, $V_1$ form a partition of $V(G_I)$. Moreover, since every edge of $G_I$ joins a vertex from $V_0$ to one from $V_1$, we have that they form a bipartition. Thus, $G_I$ is a bipartite graph.

We consider two cases.

\medskip

Case (I): There is a truth assignment satisfying all clauses $C$. Let $\Tilde{\alpha}=(\alpha_1,...,\alpha_n)$ be a truth assignment satisfying all clauses of $C$. Consider a subset $F$ of edges of $G_I$ defined as follows:
\begin{itemize}
    \item add the unique perfect matching of the path $(-1,1)-(-1, 4m)$ in Figure \ref{fig:GraphConnected} to $F$,

    \item add all edges $u_{11}u_{12}$, $u_{21}u_{22}$ to $F$, for $i=1,...,n$,

    \item if $\alpha_i=TRUE$, then add the horizontal edges of the cycle corresponding to the variable $x_i$ Figure \ref{fig:GraphCorrespondingVariable}, and if $\alpha_i=FALSE$, then add the vertical edges of the cycle corresponding to the variable $x_i$ Figure \ref{fig:GraphCorrespondingVariable}.
\end{itemize} Note that $F$ is a matching in $G_I$ covering all its vertices. Hence, it is a perfect matching. Let us compute $\nu(G_I\backslash F)$. Note that on the path $(-1,1)-(-1, 4m)$ in Figure \ref{fig:GraphConnected}, we can take $2m-1$ edges. Moreover, since all clauses are satisfied, in each copy of $G(C_j)$ ($1\leq j\leq m$) we can take eight edges. Hence,
\[\ell(G_I)\leq \nu(G_I\backslash F)=2m-1+8m=10m-1.\]

Hence, for the maximum matching $F_0$ returned by $\mathcal{B}$ we will have
\[\nu(G_I\backslash F_0)\leq (1+\varepsilon)\cdot \ell(G_I)\leq (1+\varepsilon)\cdot (10m-1). \]

\medskip

Case (II): Suppose that our instance $I$ is such that at most $\left(\frac{7}{8}+\delta\right)\cdot m$ clauses of $C$ are satisfiable by any truth assignment. Let $F$ be any perfect matching of $G_I$. Define a truth assignment $\Tilde{\gamma}=(\gamma_1,...,\gamma_n)$ as follows: $\gamma_i=TRUE$, if $F$ takes the horizontal edges of the cycle corresponding to the variable $x_i$ (Figure \ref{fig:GraphCorrespondingVariable}), and $\gamma_i=FALSE$, if $F$ takes the vertical edges of the cycle corresponding to the variable $x_i$ (Figure \ref{fig:GraphCorrespondingVariable}). We have:
\begin{equation}
    \label{eq:PerfectMatchingTruthAssignment}\nu(G_I\backslash F)=2m-1+8\cdot |SAT(\Tilde{\gamma})|+9(m-|SAT(\Tilde{\gamma})|)=11m-1-|SAT(\Tilde{\gamma})|.
\end{equation} Here $Sat(\Tilde{\gamma})$ denotes the set of clauses of $C$ satisfied by $\Tilde{\gamma}$. Since we are in Case (II), equation (\ref{eq:PerfectMatchingTruthAssignment}) implies
\[\ell(G_I)\geq 11m-1- \left(\frac{7}{8}+\delta\right)\cdot m. \]
Hence, for the perfect matching $F_0$ returned by $\mathcal{B}$ we will have
\[\nu(G_I\backslash F_0)\geq \ell(G_I)\geq 11m-1- \left(\frac{7}{8}+\delta\right)\cdot m.\]
Thus, using equality (\ref{eq:epsilondeltaPrime}), we get:
\begin{align*}
   \nu(G_I\backslash F) &\geq  11m-1- \left(\frac{7}{8}+\delta\right)\cdot m=\left(11-\frac{7}{8}-\delta\right)\cdot m-1\\
                        &=10(1+\varepsilon)\cdot m-1> 10(1+\varepsilon)\cdot m-1-\varepsilon=(1+\varepsilon)\cdot (10m-1).
\end{align*}
 Thus, using our algorithm $\mathcal{B}$ we can solve NP-hard instances from Lemma \ref{lem:3SATinapprox} by simply comparing $\nu(G_I\backslash F)$ with $(1+\varepsilon)\cdot (10m-1)$. Hence $P=NP$. The proof is complete.
\end{proof}

\section{Conclusion and future work}
\label{ConclusionSection}

In this paper, we considered the problems of construction of a maximum matching $F$ of an input graph $G$ that minimizes or maximizes the cardinality of largest matching in $G\backslash F$. We addressed these problems from the perspective of polynomial approximation. We introduced the parameters $\ell(G)$ and $L(G)$ as the minimum and maximum value of $\nu(G\backslash F)$, where $F$ is a maximum matching of $G$ and $\nu(G)$ is the matching number of $G$. In the beginning of the paper we addressed these problems from the perspective of additive approximation. Our main result states that for a small constant $c>0$, the following decision problem is NP-complete: given a graph $G$ and $k\leq \frac{|V|}{2}$, check whether there is a maximum matching $F$ in $G$, such that $|\nu(G\backslash F)-k|\leq c\cdot |V|$. Note that when $c=1$, this problem is polynomial time solvable as we observed in the beginning of the paper. 

It can be shown that in any graph $G$, we have $L(G)\leq 2\ell(G)$. Hence, any polynomial time algorithm (for example, those from \cite{Edmonds1,Edmonds2,EvenKariv,MicaliVazirani}) constructing a maximum matching of a graph is a 2-approximation algorithm for $\ell(G)$ and $\frac{1}{2}$-approximation algorithm for $L(G)$. In the paper, we complemented these observations by presenting two inapproximability results for $\ell(G)$ and $L(G)$.

From our perspective, the problem of finding the approximation threshold of $\ell(G)$ and $L(G)$ is a direction of a fruitful research.



\bibliographystyle{elsarticle-num}



\end{document}